\documentclass[10pt]{amsart}

\usepackage{amsmath}
\usepackage{amsthm}
\usepackage{amssymb}
\usepackage{amsfonts,mathrsfs}
\usepackage{stmaryrd}
\usepackage{amsxtra}  
\usepackage{epsfig}
\usepackage{verbatim}
\usepackage{xcolor}
\usepackage{graphicx}

\theoremstyle{plain}
\newtheorem{theorem}[equation]{Theorem}
\newtheorem{proposition}[equation]{Proposition}
\newtheorem{lemma}[equation]{Lemma}
\newtheorem{corollary}[equation]{Corollary}
\newtheorem{definition}[equation]{Definition}
\theoremstyle{definition}

\newtheorem{example}{Example}
\newtheorem*{remark}{Remark}
\numberwithin{equation}{section}

\newcommand{\sjump}{\hskip .2 cm}
\newcommand{\dbarstar}{\overline{\partial}^{\star}}

\newcommand{\dbar}{\overline \partial}
\newcommand{\im}{\operatorname{Im}}
\newcommand{\re}{\operatorname{Re}}
\newcommand{\supp}{\operatorname{supp}}
\newcommand{\lcap}{\operatorname{cap}}
\newcommand{\card}{\operatorname{card}}
\newcommand{\dom}{\operatorname{Dom}}

\begin{document}

\title{The closed range property for the $\dbar$-operator on
planar domains}
\author{Gallagher, A.-K.}
\date{\today}
\address{Dept. of Mathematics, Oklahoma State University,
Stillwater, OK, 74074, USA}
\email{anne.g@gallagherti.com}

\author{Lebl, J.}
\address{Dept. of Mathematics, Oklahoma State University,
Stillwater, OK, 74074, USA}
\email{lebl@okstate.edu}

\author{Ramachandran, K.}
\address{TIFR Centre For Applicable Mathematics, Tata Institute
of Fundamental Research,
Bangalore,
Karnataka, 560065, India}
\email{koushik@tifrbng.res.in}
\keywords{$\dbar$, closed range, logarithmic capacity,
Poincar\'{e}--Dirichlet inequality}
\begin{abstract}
  Let $\Omega\subset\mathbb{C}$ be an open set. 
We show that $\dbar$ has closed range in $L^{2}(\Omega)$ if and
only if the
Poincar{\'e}--Dirichlet
inequality holds. Moreover, we give necessary and sufficient
potential-theoretic conditions for
the
$\dbar$-operator to have closed range in $L^{2}(\Omega)$. We
also give a new necessary and
sufficient potential-theoretic condition for the Bergman space
of $\Omega$ to be infinite
dimensional.
\end{abstract}
\maketitle
\section{Introduction}
The purpose of this paper is to understand the closed range
property for the $\dbar$-operator for
open sets in the complex plane. That is, we study a concept,
which is fundamental for the analysis
of holomorphic functions in higher dimensions, in one complex
variable. Because of the tight
relationship of harmonic and holomorphic functions on open sets
in $\mathbb{C}$, many phenomena of
the latter may be explained and derived through potential theory
in the complex plane. Our main
result, Theorem~\ref{T:main}, is another manifestation of this
deep connection. In fact, we
completely describe the closed range property for $\dbar$ for
open sets in $\mathbb{C}$ through
two different kinds of potential-theoretic conditions.
  
  The $\dbar$-operator is 
  initially defined as 
$$\dbar f=\sum_{j=1}^{n}\frac{\partial
f}{\partial\bar{z}}_{j}d\bar{z}_{j}$$ for
  any function $f$ which is
differentiable on an open set in $\mathbb{C}^{n}$. We shall
consider the maximal $L^{2}$-extension
of $\dbar$ for the given open set. A reason for considering the
$\dbar$-operator as an
$L^{2}$-operator is that it allows one to employ Hilbert space
methods to solve the inhomogeneous
Cauchy--Riemann equations. This is of importance for the
construction of holomorphic functions in
higher dimension due to the lack of power series techniques
which are available in one complex
dimension. Note that for planar open sets, the $\dbar$-operator
may be
  identified with an extension of the derivative operator 
  $\frac{\partial}{\partial \bar{z}}$.

In this article, we give necessary and sufficient
potential-theoretic conditions for the range of $\dbar$ on an
open set $\Omega\subset\mathbb{C}$ to be closed in
$L^{2}(\Omega)$. The closed range
property is known to hold for $\dbar$ on $\Omega$ iff
  there exists a constant $C>0$ such that
 \begin{align}\label{I:CRprop0}
   \|u\|_{L^{2}(\Omega)}\leq C\|\dbar u\|_{L^{2}(\Omega)} 
 \end{align}   
for all $u\in L^{2}(\Omega)$ with $\dbar u\in L^{2}(\Omega)$ and
$u$ orthogonal to the kernel of
$\dbar$, see \cite[Theorem 1.1.1]{Hormander65}. The kernel of
$\dbar$ is the closed subspace of
$L^{2}(\Omega)$ consisting of functions
holomorphic on $\Omega$. This space is commonly called the
Bergman space and denoted by
$A^{2}(\Omega)$. Inequality \eqref{I:CRprop0} may be
reformulated as
 \begin{align}\label{I:CRprop1}
\|u\|_{L^{2}(\Omega)}\leq
C\|u_{\bar{z}}\|_{L^{2}(\Omega)}\sjump\;\;\forall\;u\perp
   A^{2}(\Omega) \text{ with } u_{\bar{z}}\in L^{2}(\Omega).
 \end{align}

The relevance of \eqref{I:CRprop1} (or \eqref{I:CRprop0}) lies
in the fact that, if the
$\dbar$-operator has closed
range for an open set in dimension greater than $1$, on two
consecutive form levels, then the
$\dbar$-Neumann operator exists as a bounded $L^{2}$-operator.
Characterizing such open sets in
higher dimensions
is an unresolved problem. A first step towards resolving this
question is to establish necessary
and sufficient
conditions for the closed range property to hold on planar open
sets.

Another point of interest of \eqref{I:CRprop1} is its formal
similarity to the
Poincar{\'e}--Wirtinger inequality. The latter is said to hold
on a domain
$\Omega\subset\mathbb{C}$, if there exists a constant $C>0$ such
that
\begin{align*}
\|v-v_{\Omega}\|_{L^{2}(\Omega)}\leq C\|\nabla
v\|_{L^{2}(\Omega)}
\end{align*}
for all $v$ in $H^{1}(\Omega)$, the $L^{2}$-Sobolev-$1$-space of
$\Omega$. Here, $v_{\Omega}$ is
the average value of $v$
on $\Omega$. 
Since the kernel, $\ker\nabla$, of $\nabla$ is either the set of
constants or trivial, it follows
that $v-v_{\Omega}$ is orthogonal to $\ker\nabla$. In fact,
$u\in
L^{2}(\Omega)\cap\left(\ker\nabla\right)^{\perp}$ iff
$u_{\Omega}=0$. Thus, the
Poincar{\'e}-Wirtinger inequality is
\begin{align*}
\|u\|_{L^{2}(\Omega)}\leq C\|\nabla
u\|_{L^{2}(\Omega)}\;\;\sjump\forall\;u\in H^{1}(\Omega)\cap
\left(\ker\nabla\right)^{\perp}.
\end{align*}
Hence the closed range property of $\dbar$ may be considered a
Poincar{\'e}--Wirtinger
inequality for $\dbar$. 

It turns out that the closed range property for $\dbar$ is more
closely related to the
Poincar{\'e}--Dirichlet inequality. That is, the inequality
\begin{align*}
\|v\|_{L^{2}(\Omega)}\leq C\|\nabla
v\|_{L^{2}(\Omega)}\;\;\sjump\forall\;v\in
H_{0}^{1}(\Omega)\end{align*}
for some constant $C>0$; here $H_{0}^{1}(\Omega)$ is the
completion of
$\mathcal{C}^{\infty}_{c}(\Omega)$
with respect to the Sobolev-$1$-norm. At first, this might seem
surprising as membership to the
domain of
$\dbar$ has no boundary condition folded in. However, the domain
of the Hilbert space adjoint,
$\dbarstar$,
of $\dbar$ is contained in $H_{0}^{1}(\Omega)$. Due to the
Closed Range Theorem of
Banach, the $\dbar$-operator has closed range if and only if its
Hilbert space adjoint does. So it
might be less
surprising that the Poincar{\'e}--Dirichlet inequality is in
fact equivalent to $\dbar$ having
closed range in
$L^{2}(\Omega)$, see Theorem \ref{T:main} below.

\medskip

To describe the closed range property for $\dbar$ on planar open
sets in potential-theoretic
terms, we use the
notion of logarithmic capacity of a set in the complex plane. We
denote the logarithmic capacity
of a set
$E\subset\mathbb{C}$ by $\lcap{E}$; see
Section~\ref{SS:potentialtheory} for the definition.
Following nomenclature used in describing sufficiency conditions
for the Poincar{\'e}--Dirichlet
inequality, see,
e.g., \cite[\S 2]{Sou99},\cite[Proposition 2.1]{Sou00}, and
references therein, we introduce the
following
terminology. For a set $\Omega\subset\mathbb{C}$, define the
\emph{capacity inradius} of $\Omega$
by
$$\rho_{\lcap}(\Omega)=\sup\left\{R\geq0:\;\forall\;\delta>0\;\exists
\;z\in\mathbb{C}\text{ such
that }
\lcap\left(\mathbb{D}(z,R)\cap\Omega^{c}\right)<\delta\right\},$$
see Section~\ref{SS:potentialtheory} for more details on this
concept. Finiteness of the capacity
inradius completely
characterizes those open sets for which $\dbar$ has closed
range:

\pagebreak[3]

\begin{theorem}\label{T:main}
Let $\Omega\subset\mathbb{C}$ be an open set. Then the following
are equivalent:
  \begin{itemize}
    \item[(1)]  $\dbar$ has closed range in $L^{2}(\Omega)$.
\item[(2)] The Poincar{\'e}--Dirichlet inequality holds on
$\Omega$.
       \item[(3)] $\rho_{\lcap}(\Omega)<\infty$.
\item[(4)] There exists a bounded function
$\varphi\in\mathcal{C}^{\infty}(\Omega)$ and a constant
    $c>0$ such that
    $\triangle\varphi(z)>c$ holds for all $z\in\Omega$.
  \end{itemize}
\end{theorem}

The implication ``(4)$\Rightarrow$(1)'' is a consequence of work
by H\"ormander in \cite[Theorem
2.2.1$'$ ]{Hormander65}, see also \cite[Corollary
6.11]{HerMcN16}. Our proof of
``(3)$\Rightarrow$(4)'' is constructive. In
fact, the
function $\varphi$ in (4) is built from a sequence of potential
functions associated to the
equilibrium measures of
 certain  compact sets in the complement of the open set.

The idea for the proof of ``(3)$\Rightarrow$(4)'' lead us to the
completion of the
characterization of planar open sets with infinite dimensional
Bergman spaces in terms of the
existence of bounded,
strictly subharmonic functions, see (4) in the following
theorem.

\begin{theorem}\label{T:Bergmanspace}
Let $\Omega\subset\mathbb{C}$ be an open set. Then the following
are equivalent:
  \begin{itemize}
    \item[(1)] $A^{2}(\Omega)\neq\{0\}$.
    \item[(2)] $\dim A^{2}(\Omega)=\infty$.
    \item[(3)] $\lcap(\Omega^{c})>0$.
\item[(4)] There exists a bounded function
$\varphi\in\mathcal{C}^{\infty}(\Omega)$ such that
    $\triangle\varphi(z)>0$ for all $z\in \Omega$.
  \end{itemize}
\end{theorem}
The equivalence of (1) and (2) is shown by Wiegerinck in
\cite{Wie84}, the equivalence of (1) and
(3) by Carleson in \cite[Theorem 1.a in \S VI]{Car67}, the
implication ``(4)$\Rightarrow$(2)'' by
Harz, Herbort, and the first author in \cite{GaHaHe17}. It
follows from ``(3)$\Rightarrow$(4)''
and from the proof of ``(4)$\Rightarrow$(2)'' that $A^{2}(\Omega)$ is
a separating set for $\Omega$ iff
it is non-trivial.

The paper is structured as follows. We define basic notions of
the $L^{2}$-theory for $\dbar$ and
potential theory for open sets in the complex plane in Section
\ref{S:prelims}. In this section,
we also recall the connection between the best constant in the
Poincar{\'e}--Dirichlet inequality
and the lowest eigenvalue of the Dirichlet--Laplacian. Moreover,
we derive basic characteristics
of the closed range property of $\dbar$ and conclude the section
with a proof of the equivalence
of the closed range property for $\dbar$ and the
Poincar{\'e}--Dirichlet inequality. Section
\ref{S:necessity} contains the proof of the implication
``(1)$\Rightarrow$(3)'' of Theorem
\ref{T:main}. We first give a proof of this implication under an
additional assumption, since it
is based on standard $\dbar$-arguments that indicate how to
approach the higher dimensional case.
The general proof, also in Section \ref{S:necessity}, is based
on the connection of the closed
range property to the Poincar{\'e}--Dirichlet inequality for
bounded open sets, a solution to the
(lowest) eigenvalue problem for the Dirichlet--Laplacian on the
unit disc, and so-called
$r$-logarithmic potentials. The proofs of
``(3)$\Rightarrow$(4)'' of
Theorem~\ref{T:main} and ``(3)$\Rightarrow$(4)'' of Theorem
\ref{T:Bergmanspace} are given in
Sections \ref{S:sufficiency} and
\ref{S:Bergmanspace}, respectively. Both are constructive and
based on using potential functions
associated to certain compact sets in the complement of the open
set in consideration.


\section{Preliminaries}\label{S:prelims}
\subsection{The $\dbar$-operator and its closed range property
on open sets
in $\mathbb{C}$}\label{SS:dbarstuff}
For an open set $\Omega\subset\mathbb{C}$, we denote by
$\mathcal{C}^{\infty}(\Omega)$ and
$\mathcal{C}^{\infty}_{c}(\Omega)$ the family of smooth
functions on $\Omega$ and the family of
smooth functions on $\Omega$ whose (closed) support is compact
in $\Omega$, respectively. As
usual, $L^{2}(\Omega)$ is the space of square-integrable
functions on $\Omega$, the associated
norm and inner product are denoted by $\|.\|_{L^{2}(\Omega)}$
and $(.,.)_{L^{2}(\Omega)}$,
respectively. The $L^{2}$-Sobolev-$1$-space, $H^{1}(\Omega)$, on
$\Omega$ is the subspace of
functions $f\in L^{2}(\Omega)$ for which the norm
$$\|f\|_{H^{1}(\Omega)}=\left(\|f\|_{L^{2}(\Omega)}^{2}+\|\nabla
f \|_{L^{2}(\Omega)}^{2}
\right)^{\frac{1}{2}}$$ is finite.
Here $\nabla f$ is meant in the sense of distributions.
$H_{0}^{1}(\Omega)$ is the closure of
$\mathcal{C}^{\infty}_{c}(\Omega)$ with respect to
$\|.\|_{H^{1}(\Omega)}$.

The $\dbar$-operator on $\Omega$ is defined as $\dbar
u=u_{\bar{z}}\,d\bar{z}$ for any
$u\in\mathcal{C}^{\infty}(\Omega)$. Since $(0,1)$-forms on
$\Omega$ may be identified with
functions on $\Omega$, we henceforth identify $\dbar u$ with
$u_{\bar{z}}$. The maximal extension
of the $\dbar$-operator, still denoted by $\dbar$, is defined as
follows: we first allow $\dbar$
to act on functions in $L^{2}(\Omega)$ in the sense of
distributions and then restrict its domain
to those functions whose image under $\dbar$ lies in
$L^{2}(\Omega)$. That is,
$$
\dom(\dbar)=\left\{u\in L^{2}(\Omega): \dbar u\text{ (in the
sense of distributions) in }
L^{2}(\Omega)\right\}.
$$
As $\mathcal{C}^{\infty}_{c}(\Omega)$ is dense in
$L^{2}(\Omega)$ with respect to the
$L^{2}(\Omega)$-norm, it follows that $\dbar$ is a densely
defined operator on $L^{2}(\Omega)$;
moreover, it is a closed operator. To define the Hilbert space
adjoint, $\dbarstar$, of $\dbar$ we
first define its domain $\dom(\dbarstar)$ to be the space of
those $v\in L^{2}(\Omega)$ for which
there exists a positive constant $C=C(v)$ such that $$|(\dbar
u,v)_{L^{2}(\Omega)}|\leq
C\|u\|_{L^{2}(\Omega)}
  \sjump\;\;\forall\;u\in\dom(\dbar),
$$
i.e., for any $v\in\dom(\dbarstar)$, the map $u\mapsto (\dbar
u,v)_{L^{2}(\Omega)}$ is a bounded
linear functional on $\dom(\dbar)$. Hence, by Hahn--Banach, the
map extends to a bounded linear
functional on $L^{2}(\Omega)$.
It then follows from the Riesz Representation theorem that for
any $v\in\dom(\dbarstar)$ there
exists a $w\in L^{2}(\Omega)$ such that $$(\dbar
u,v)_{L^{2}(\Omega)}=(u,w)_{L^{2}(\Omega)}\sjump\;\;\forall
\;u\in\dom(\dbar).$$ Set $\dbarstar
v=w$. If $\Omega$ has smooth boundary, it follows from an
integration by parts argument, that
whenever
$v\in\dom(\dbarstar)\cap\mathcal{C}^{\infty}(\overline{\Omega})$,
then $v|_{b\Omega}=0$
and $\dbarstar v=-v_{z}$. Furthermore the following density
result holds.
\begin{lemma}\label{L:Straube}
Let $\Omega\subset\mathbb{C}$ be an open set, then
$\mathcal{C}^{\infty}_{c}(\Omega)$ is dense in
$\dom(\dbarstar)$ with respect to the graph norm
$$v\mapsto\left(\|v\|_{L^{2}(\Omega)}^{2}+\|\dbarstar
v\|_{L^{2}(\Omega)}^{2} \right)^{1/2}.$$
\end{lemma}
This density result could be expected considering that elements
of $\dom(\dbarstar)$ in some sense
vanish on the boundary while the above graph norm restricted to
$\mathcal{C}^{\infty}_{c}(\Omega)$
is equivalent to $\|.\|_{H^{1}(\Omega)}$, see the argument in
the proof of Lemma
\ref{L:closedrangecompactsupport}.
A concise proof of Lemma~\ref{L:Straube} may be found in
\cite[Proposition 2.3]{Straube10}. For
the convenience of the reader, we reproduce the proof here.
\begin{proof}
Let $u\in\dom(\dbarstar)$ be given. Suppose $u$ is orthogonal to
all functions on
$\mathcal{C}_{c}^{\infty}(\Omega)$ with respect to the inner
product associated to the graph norm,
i.e.,
  \begin{align*}
(u,v)_{L^{2}(\Omega)}+(\dbarstar u,\dbarstar
v)_{L^{2}(\Omega)}=0
    \;\;\sjump\forall\;v\in\mathcal{C}^{\infty}_{c}(\Omega).
  \end{align*}
If this forces $u$ to be zero, then the claim follows. Note
first that
$(\dbarstar u,\dbarstar v)_{L^{2}(\Omega)}$,
$v\in\mathcal{C}^{\infty}_{c}(\Omega)$,
defines $\dbar\dbarstar u$ in the sense of distributions. In
particular, it follows that
$u+\dbar\dbarstar u$ is
zero as a distribution. As $u\in L^{2}(\Omega)$, it then follows
that $\dbar\dbarstar u\in
L^{2}(\Omega)$.
  Thus
  $(u+\dbar\dbarstar u,u)_{L^{2}(\Omega)}=0$,
hence $$\|u\|^{2}_{L^{2}(\Omega)}+\|\dbarstar
u\|^{2}_{L^{2}(\Omega)}=0.$$
  Therefore $u=0$, which proves the claim.
\end{proof}

\medskip

The next proposition gives basic, equivalent descriptions for
the $\dbar$-operator to have
 closed range, which is the property that whenever 
$\{\dbar u_{n}\}_{n\in\mathbb{N}}$ converges in $L^{2}(\Omega)$
for $\{u_{n}\}_{n\in\mathbb{N}}
\subset\dom(\dbar)$, then $\|\dbar u_{n}-\dbar u\|_{L^2(\Omega)}$ goes to $0$ as $n\to\infty$
for some $u\in\dom(\dbar)$.

\begin{proposition}\label{P:CRpropeq}
Let $\Omega\subset\mathbb{C}$ be an open set. Then the following
are equivalent:
  \begin{itemize}
    \item[(i)] $\dbar$ has closed range in $L^{2}(\Omega)$.
  
    \item[(ii)]  There exists a constant $C>0$ such that
    $\|u\|_{L^{2}(\Omega)}\leq C\|\dbar u\|_{L^{2}(\Omega)}$
    holds for all 
    $u\in\dom(\dbar)$ with $u\perp\ker\dbar$.
  
    \item[(iii)] There exists a constant $C>0$ such that 
$\|v\|_{L^{2}(\Omega)}\leq C\|\dbarstar v\|_{L^{2}(\Omega)}$    holds for all 
    $v\in \dom(\dbarstar)$ with $v\perp\ker \dbarstar$.
   
\item[(iv)] There exists a constant $C>0$ such that for all $f$    in the range of $\dbar$
    there exists a $v\in L^2(\Omega)$
such that $\dbar v=f$ holds in the distributional sense and $$\|v\|_{L^{2}(\Omega)}\leq
C\|f\|_{L^{2}(\Omega)}.$$
   \end{itemize}
\end{proposition}
These equivalences are well-known, and, in fact, higher
dimensional analogs of (i)--(iv) are true.
For the convenience of the reader, we give either references or
short arguments for the proofs of
Proposition \ref{P:CRpropeq}.

\begin{proof}
  The equivalences of (i)--(iii) are proved 
   in \cite[Theorem~1.1.1]{Hormander65}. 
 
To see that (iv) implies (ii), let $u\in\dom(\dbar)$ with
$u\perp\ker\dbar$ be given. By (iv)
there exists a $v\in L^2(\Omega)$ such that
 $\dbar v=\dbar u$ in the distributional sense and  
 $$\|v\|_{L^{2}(\Omega)}\leq C\|\dbar u\|_{L^{2}(\Omega)}.$$
 Note that for any $h\in A^2(\Omega)$,  $w=v+h$ also satisfies 
$\dbar w=f$ in the distributional sense. Since $A^2(\Omega)$ is
a closed subspace of
$L^2(\Omega)$, it follows that there exists a
$L^2(\Omega)$-minimal solution $w_0$ to the
$\dbar$-equation with data $f$. It is easy to show that $w_0$ is
orthogonal to $A^2(\Omega)$.
 Since $\dbar(u-w_0)=0$ in the distributional 
 sense, it follows from the ellipticity of $\dbar$ 
on functions that $u-w_0\in A^{2}(\Omega)$. Since $u-w_0$ is
also in $(A^{2}(\Omega))^{\perp}$, it follows that $u=w_0$ so
that (ii) holds.
  
The implication (iii)$\Rightarrow$(iv) follows from a standard
duality argument, see Theorem~1.1.4 in
\cite[Theorem~1.1.4]{Hormander65}
with $A=\frac{1}{C}\text{Id}$, $T=\dbar$, $S=0$,
$H_{1}=L^{2}(\Omega)=H_{2}$, $H_{3}=0$ and
  $F=(\ker\dbarstar)^{\perp}$. 
\end{proof}

The constants in (ii)--(iv) of Proposition \ref{P:CRpropeq} may
be chosen to be the same. For the
best possible constant, we introduce the following notation.

\begin{definition}
  Let $\Omega$ be an open set in  $\mathbb{C}$. 
Then $\dbar$ is said to have closed range in $L^{2}(\Omega)$
with
  constant $C$ if
  $$\|u\|_{L^{2}(\Omega)}\leq C\|\dbar u\|_{L^{2}(\Omega)}$$
holds for all $u\in \dom(\dbar)\cap A^{2}(\Omega)^{\perp}$. In
that case, set
  $$\mathfrak{C}(\Omega)
=\inf\left\{C:\|u\|_{L^{2}(\Omega)}\leq C\|\dbar
u\|_{L^{2}(\Omega)} \;\;
\forall\;u\in\dom(\dbar) \cap A^{2}(\Omega)^{\perp}\right\}.$$ If the closed range property for
$\dbar$ does not hold in
    $L^{2}(\Omega)$, we say that $\mathfrak{C}(\Omega)=\infty$.
\end{definition} 

\medskip


\subsection{Terminology from potential theory in the
plane}\label{SS:potentialtheory}
Let $\mu$ be a finite Borel measure with compact support in
$\mathbb{C}$. The potential,
$p_{\mu}$, associated to $\mu$ is defined by
$$p_{\mu}(z)=\int_{\mathbb{C}}\ln|z-w|\;d\mu(w).$$
The energy, $I_{\mu}$, of $\mu$ is given by
$$I_{\mu}=\int_{\mathbb{C}}p_{\mu(z)}\;d\mu(z)=\int_{\mathbb{C}}\int_{\mathbb{C}}\ln|z-w|\;d\mu(w)\;d\mu(z).$$

A set $E\subset\mathbb{C}$ is called polar if the energy of
every non-trival, finite Borel measure
with compact support in $E$ is $-\infty$.
If for a compact set $K\subset\mathbb{C}$, there is a finite
Borel probability measure $\nu$ with
support in $K$ such that $$I_{\nu}=\sup\{I_{\mu}:\mu \text{
finite Borel probability measure with
support in } K\},$$
then $\nu$ is said to be an equilibrium measure for $K$. Any
compact set has an equilibrium
measure, see, e.g., \cite[Theorem 3.3.2]{Rans95}. Moreover, this
equilibrium measure is unique for
any non-polar compact set, see \cite[Theorem 3.7.6]{Rans95}. The
logarithmic capacity of a
compact, non-polar set $K$ is defined as $\lcap{K}=e^{I_{\nu}}$,
where $\nu$ is the equilibrium
measure of $K$. If $K$ is compact and polar, then $\lcap(K)=0$.
For a general set
$E\subset\mathbb{C}$, the logarithmic capacity, $\lcap(E)$, of
$E$ is defined as $\sup
e^{I_{\mu}}$ for $\mu$ is a finite Borel probability measure
with compact support in $E$. Note
that a set $E$ is polar if and only if $\lcap(E)=0$.

That the notion of positive logarithmic capacity comes into play
for the description of the
dimension of the Bergman space can be seen through the following
observation. Let
$K\subset\mathbb{C}$ be a compact, non-polar set, and $\mu$ the
associated equilibrium measure.
Then $p_{\mu}$ is a non-constant function, which is harmonic on
$K^{c}$, and bounded from below by
$\ln(\lcap(K))$ by Frostman's theorem. Thus $e^{-p_{\mu}}$ is a
bounded, smooth, subharmonic,
non-harmonic function on $K^{c}$. Hence it is a good candidate
for the construction of subharmonic
functions in part (4) of \ref{T:Bergmanspace}.

This construction is also used in the proof of the necessity of
the existence of bounded, strictly
subharmonic functions for the closed range property to hold for
$\dbar$, see part (4) of
Theorem~\ref{T:main}. To achieve this strict subharmonicity we
need compact sets, contained in the
complement, and of sufficiently large logarithmic capacity, to
be somewhat regularly distributed
over the complex plane. This vague description can be made
precise using the terminology of
capacity inradius as introduced in the first section.
Recall that for a set $\Omega\subset\mathbb{C}$, the capacity
inradius of $\Omega$ is defined by
$$\rho_{\lcap}(\Omega)=\sup\left\{R\geq0:\;\forall\;\delta>0\;\exists
\;z\in\mathbb{C}\text{ such
that }
\lcap\left(\mathbb{D}(z,R)\cap\Omega^{c}\right)<\delta\right\}.$$
Note that finiteness of the capacity inradius of $\Omega$ means
that for any
$M>\rho_{\lcap}(\Omega)$ there is a $\delta>0$ such that for any
point in $z\in\Omega$ there is a
set in the complement of $\Omega$, whose logarithmic capacity is
larger than $\delta$ while its
distance to $z$ is less than $M$. For instance, both
$\rho_{\lcap}(\mathbb{C})$ and
$\rho_{\lcap}\left(\mathbb{C}\setminus
\left(\mathbb{Z}+\sqrt{-1}\,\mathbb{Z})\right)\right)$ are
infinite. However, if for given $\epsilon>0$, $K_{j,\ell}$ is
the disc of radius $\epsilon$
centered at $j+\sqrt{-1}\,\ell$ or a line segment of length
$\epsilon$ containing
$j+\sqrt{-1}\,\ell$, then
$\rho(\mathbb{C}\setminus\bigcup_{j,\ell\in\mathbb{Z}
}K_{j,\ell})$ is
finite. We note that in the case of the removed discs, the
Poincar{\'e}--Dirichlet inequality is
known to hold, see part (ii) in Proposition 2.1 in \cite{Sou00}
and references therein. It appears
to be new that, as a consequence of Theorem \ref{T:main}, the
Poincar{\'e}--Dirichlet inequality
is true in the removed line segments case as well.


\subsection{The Poincar{\'e}--Dirichlet inequality}
Let $\Omega\subset\mathbb{C}$ be an open set. 
The Poincar{\'e}--Dirichlet inequality is said to hold on
$\Omega$, if there exists a constant
$C>0$ such that
\begin{align}\label{E:Poincare}
\|v\|_{L^{2}(\Omega)}\leq C\|\nabla
v\|_{L^{2}(\Omega)}\;\;\sjump\;\forall\;v\in
H_{0}^{1}(\Omega).
\end{align}
Whenever \eqref{E:Poincare} holds, it is customary to consider
\begin{align}\label{E:Rayleigh}
\lambda_{1}(\Omega)=\min\Biggl\{\frac{\|\nabla\varphi\|^{2}_{L^{2}(\Omega)}}{\|\varphi\|^{2}_{L^{2}(\Omega)}}:
\varphi\in H_{0}^{1}(\Omega),\;\varphi\not\equiv 0\Biggr\}.
\end{align}

This notation stems from the fact that $\lambda_{1}(\Omega)$ is
the smallest eigenvalue for the
Dirichlet--Laplacian.
In fact, if $\lambda_{1}(\Omega)$ in \eqref{E:Rayleigh} is
attained at $\psi\in
H_{0}^{1}(\Omega)$, then, for fixed $\varphi\in
H_{0}^{1}(\Omega)$, the function
$$f_{\varphi}(t)=\frac{\left\|\nabla\left(\psi+t\varphi
\right)\right\|_{L^{2}(\Omega)}^{2}}{\left\|\psi+t\varphi\right\|_{L^{2}(\Omega)}^{2}}$$
is differentiable near $t=0$ and has a critical point there.
Unraveling the equation
$f_{\varphi}'(0)=0$
for all $\varphi\in H_{0}^{1}(\Omega)$ then leads to observing
that
$\psi$ is a distributional solution to the boundary value
problem
\begin{align}\label{E:DirichletLaplacian}
 \begin{cases} 
      \triangle\psi+\lambda\psi=0\text{ on }\Omega\\
      \psi|_{b\Omega}=0 
   \end{cases}
\end{align}
for $\lambda=\lambda_{1}(\Omega)$. 

Furthermore, we note that, if $\Omega$ has smooth boundary and
$\psi\in\mathcal{C}^{\infty}(\overline{\Omega})$ with $\psi=0$
on $bD$,
then integration by parts yields
\begin{align}\label{E:RayleighLaplace}
\lambda_{1}(\Omega)\leq
\frac{\|\nabla\psi\|^{2}_{L^{2}(\Omega)}}{\|\psi\|^{2}_{L^{2}(\Omega)}}
=\frac{(-\triangle\psi,\psi)_{L^{2}(\Omega)}}{\|\psi\|^{2}_{L^{2}(\Omega)}}
  \leq
\frac{\|\triangle\psi
\|_{L^{2}(\Omega)}}{\|\psi\|_{L^{2}(\Omega)}}.
\end{align}

\medskip


\subsection{Basic characteristics of the closed range property
for $\dbar$}
\begin{proposition}\label{P:basicprops}
  Let $\Omega\subset\mathbb{C}$ be an open set.
  \begin{itemize}
\item[(a)] Invariance under rigid transformations: If $\Omega'$
is obtained from $\Omega$ by
translation, rotation, reflection, or a combination thereof,
then
$\mathfrak{C}(\Omega')=\mathfrak{C}(\Omega)$.
\item[(b)] Linearity under scaling:
$r\mathfrak{C}(\Omega)=\mathfrak{C}(r\Omega)$ for
  $r\Omega=\{rz:z\in\Omega\}$ and $r>0$.
\item[(c)] Invariance under polar sets:
$\mathfrak{C}(\Omega)=\mathfrak{C}(\Omega')$ for any open
set
$\Omega'\subset\Omega$ such that $\Omega\setminus\Omega'$ is
polar.
   \end{itemize}
\end{proposition}

\begin{remark}
(i) If $\Omega\subset\mathbb{C}$ is a bounded open set, then
$\mathfrak{C}(\Omega)<\infty$
   by the implication 
``(4)$\Rightarrow$(1)'' Theorem \ref{T:main} with, say,
$\varphi(z)=|z|^{2}$, see \cite[Theorem
2.2.1$'$]{Hormander65}.
  
 (ii) Note that  (b) implies that 
$$\mathfrak{C}(\mathbb{C})=\mathfrak{C}(r\mathbb{C})=r\mathfrak{C}(\mathbb{C})$$
for all $r>0$. Hence either $\mathfrak{C}(\mathbb{C})$ equals
$0$ or $\infty$. If
$\mathfrak{C}(\mathbb{C})$
was $0$, then by (ii) of Proposition \ref{P:CRpropeq} it would
follow that
$\|\nabla u\|_{L^2(\mathbb{C})}=2\|u_z\|_{L^2(\mathbb{C})}=0$
for all
$u\in\mathcal{C}_c^{\infty}(\mathbb{C})$. As this is clearly not
true, it follows that
$\mathfrak{C}(\mathbb{C})=\infty$, i.e., $\dbar$ does not have
closed range in $L^2(\mathbb{C})$.

Since $\mathbb{Z}+\sqrt{-1}\,\mathbb{Z}$ is a polar set, (c)
implies that $\dbar$ for
$$\Omega=\mathbb{C}\setminus\left(\mathbb{Z}+\sqrt{-1}\,\mathbb{Z}\right)$$
  also does not have closed range in $L^2(\Omega)$.
  
(iii) An argument similar to the one given in (ii) yields that
$\dbar$ for the upper half plane
does not have closed
range. Since $\mathfrak{C}(\mathbb{D}(0,1))$ is finite by (i),
it follows that the closed range
property is not invariant
  under biholomorphic equivalences.
\end{remark}

\begin{proof}
Translations and rotations are biholomorphic maps for which the
absolute value of its Jacobian is
$1$.
Hence, invariance readily follows. Any reflection may be written
as a composition of translations,
rotations, and
complex conjugation. So it remains to show the invariance under
complex conjugation. Denote
complex conjugation by $T$, i.e., $Tz=\bar{z}$ for
$z\in\mathbb{C}$. Write
$\Omega_{T}=\{z\in\mathbb{C}:Tz\in \Omega\}$. Observe that the
map $u\mapsto T\circ u\circ T$
yields an isometry of $L^{2}(\Omega)$ and $L^{2}(\Omega_{T})$ as
well as $A^{2}(\Omega)$ and
$A^{2}(\Omega_{T})$. Further, one easily verifies that
$\dbar(T\circ u\circ T)=T\circ\dbar u\circ
T$. Hence $\dbar$ has closed range in $L^{2}(\Omega)$ iff it has
closed range in
$L^{2}(\Omega_{T})$, and
$\mathfrak{C}(\Omega)=\mathfrak{C}(\Omega_{T})$.
   
Part (b) follows straightforwardly from the fact that
$\dbar(u(rz))=r(\dbar u)(rz)$ for any scalar
$r$.

  To prove part (c), note first that 
  $u$, $\dbar u\in L^{2}(\Omega')$ whenever 
$u$, $\dbar u\in L^{2}(\Omega)$, since $\Omega\setminus\Omega'$
is of Lebesgue measure zero.
   Moreover, since $\Omega\setminus\Omega'$ 
  is polar, 
$A^{2}(\Omega)=A^{2}(\Omega')$, see e.g., part (c) of Theorem
9.5 in \cite{Con95}. Hence
$A^{2}(\Omega)^{\perp}=A^{2}(\Omega')^{\perp}$. Therefore,
$\mathfrak{C}(\Omega)=\mathfrak{C}(\Omega')$.
\end{proof}
 
 \begin{proposition}\label{P:Poincare}
Let $\Omega\subset\mathbb{C}$ be an open set. The
Poincar{\'e}--Dirichlet inequality holds on
$\Omega$ if and only
   if $\dbar$ has 
   closed range in $L^{2}(\Omega)$.
Moreover,
$$\mathfrak{C}(\Omega)=\frac{2}{\sqrt{\lambda_{1}(\Omega)}}.$$
 \end{proposition}
 
In our proof of Proposition \ref{P:Poincare} the following two
lemmata are essential.
 
 \begin{lemma}\label{L:closedrangecompactsupport}
Let $\Omega\subset\mathbb{C}$ be an open set. If the
Poincar{\'e}--Dirichlet inequality holds,
then
  \begin{align*}
    \left\|\varphi\right\|_{L^{2}(\Omega)}\leq 
\frac{2}{\sqrt{\lambda_{1}(\Omega)}}
\left\|\dbarstar\varphi\right\|_{L^{2}(\Omega)}
\sjump\;\;\forall\; \varphi\in \mathcal{C}^{\infty}_{c}(\Omega).  \end{align*}
 \end{lemma}

\begin{proof}[Proof of Lemma \ref{L:closedrangecompactsupport}]
  Let $\Omega\subset\mathbb{C}$ be an open set. Then for 
$\varphi\in\mathcal{C}_{c}^{\infty}(\Omega)$, it follows from
integration by parts that
  \begin{align*}
      \left\|\nabla\varphi\right\|_{L^{2}(\Omega)}^{2}
    =(-\triangle\varphi,\varphi)_{L^{2}(\Omega)}
    =4(\dbar\dbarstar\varphi,\varphi)_{L^{2}(\Omega)}=
    4\|\dbarstar\varphi\|_{L^{2}(\Omega)}^{2}.
  \end{align*}
\end{proof}

\begin{lemma}\label{L:H01notdense}
Let $\Omega\subset\mathbb{C}$ be an open set. Then non-trivial
constant functions are not
contained in $H_0^1(\Omega)$.
\end{lemma}
\begin{proof}
  For a function $f\in H_0^1(\Omega)$, write $\widehat{f}$ to denote the function obtained from extending $f$  to be $0$ outside of $\Omega$.  
  Note that 
   $\widehat{f}\in H^1(\mathbb{R}^2)$.
  In fact, since $f\in H_0^1(\Omega)$, there exists a $\{\phi_n\}_{n\in\mathbb{N}}\subset\mathcal{C}_c^{\infty}(\Omega)$ such that $\phi_n
  \longrightarrow f$  in the Sobolev-$1$-norm on $\Omega$ as $n\to\infty$. As $\{\phi_n\}_{n\in\mathbb{N}}$ may be considered a subset of 
  $\mathcal{C}_c^{\infty}(\mathbb{R}^2)$, 
  we get that $\{\phi_n\}_{n\in\mathbb{N}}$ is a Cauchy sequence in the Sobolev-$1$-norm on $\mathbb{R}^2$. 
  It then follows that  $\phi_n\longrightarrow \widehat{f}$  in the Sobolev-$1$-norm on $\mathbb{R}^2$ as $n\to\infty$.
  
  Now suppose $f\equiv c\neq 0$. Then $\widehat{f}=c\chi_\Omega$, where $\chi_\Omega$ is the characteristic function of $\Omega$. Fubini's theorem
  yields that there is a set $\mathcal{N}\subset\mathbb{R}$ of $1$-dimensional measure $0$,
   such that $\widehat{f}(.,y_0)\in H^1(\mathbb{R})$ for all $y_0\notin\mathcal{N}$. This implies that $f(.,y_0)$ is equal to a continuous function almost everywhere in $\mathbb{R}$
 for all $y_0\notin\mathcal{N}$. Since $\Omega$ is open, it follows that there exists a $y_0\notin\mathcal{N}$   
  such that the intersection of the line $y=y_0$ with $\Omega$ is of positive one-dimensional Lebesgues measure. Thus $\widehat{f}(.,y_0)=c$ almost 
  everywhere. This is a contradiction as $c\notin L^2(\mathbb{R})$.
\end{proof}

We are now set to prove Proposition \ref{P:Poincare}.
 
 \begin{proof}[Proof of Proposition \ref{P:Poincare}]
Suppose the Poincar{\'e}--Dirichlet inequality holds. The proof for $\dbar$ having closed range is based on
Lemmata~\ref{L:closedrangecompactsupport}
and \ref{L:Straube}.
To wit, let $u\in\dom(\dbarstar)\cap(\ker\dbarstar)^{\perp}$ and
$\epsilon>0$ be given. Then by
the density result, there exists a
$\varphi\in\mathcal{C}^{\infty}_{c}(\Omega)$ such that
$$\left(\|u-\varphi\|^{2}_{L^{2}(\Omega)}+\|\dbarstar
u-\dbarstar\varphi\|_{L^{2}(\Omega)}^{2}\right)<\epsilon.$$
Hence using Lemma \ref{L:closedrangecompactsupport}, after an
application of the triangle inequality,
yields
 \begin{align*}
\|u\|_{L^{2}(\Omega)}\leq\|u-\varphi\|_{L^{2}(\Omega)}+\|\varphi\|_{L^{2}(\Omega)}
\leq
\epsilon+\frac{2}{\sqrt{\lambda_{1}(\Omega)}}\|\dbarstar\varphi\|_{L^{2}(\Omega)}.
 \end{align*}
 It then follows from the density result that
 \begin{align}
\|u\|_{L^{2}(\Omega)}\leq
\epsilon\left(1+\frac{2}{\sqrt{\lambda_{1}(\Omega)}}\right)+\frac{2}{\sqrt{\lambda_{1}(\Omega)}}\|\dbarstar
u\|_{L^{2}(\Omega)},
 \end{align}
Since $\epsilon$ was chosen arbitrarily, it follows from
Proposition \ref{P:CRpropeq}, that
$\dbar$ has closed range in $L^{2}(\Omega)$ and
$\mathfrak{C}(\Omega)\leq\frac{2}{\sqrt{\lambda_{1}(\Omega)}}$. 

Now suppose that $\dbar$ has closed range in $L^2(\Omega)$ with constant $C$. Note first that
the adjoint operator 
$\dbarstar$ is closed because $\dbar$ is. Hence $\ker\dbarstar$ is 
closed in $L^2(\Omega)$, and we obtain the orthogonal decomposition:
$$L^2(\Omega)=\ker\dbarstar \oplus (\ker\dbarstar)^{\perp}.$$
Now let $\phi\in\mathcal{C}^\infty_c(\Omega)$. By the above decomposition of 
$L^2(\Omega)$, $\phi$ may 
be written as $\phi_1+\phi_2$ for some $\phi_1\in(\ker\dbarstar)^\perp$ and 
$\phi_2\in\ker\dbarstar$. Note that both $\phi$ and $\phi_2$ are in $\dom(\dbarstar)$. 
Since the latter is a vector space, it follows that $\phi_1\in\dom(\dbarstar)$ as well.
Next, by assumption $\dbar$ has closed range in $L^2(\Omega)$. Therefore
\begin{align*}
  \|\phi\|_{L^2(\Omega)}^2= \|\phi_1\|_{L^2(\Omega)}^2
  + \|\phi_2\|_{L^2(\Omega)}^2
 & \leq C^2\|\dbarstar\phi_1\|^2_{L^2(\Omega)}+ \|\phi_2\|_{L^2(\Omega)}^2\\
  &= C^2\|\dbarstar\phi\|^2_{L^2(\Omega)}+ \|\phi_2\|_{L^2(\Omega)}^2.
\end{align*}
We shall employ Lemmata \ref{L:Straube} and \ref{L:H01notdense} to show that 
$\phi_2=0$ almost everywhere. Note first that by Lemma \ref{L:Straube}, there exists a sequence 
$\{f_n\}\subset\mathcal{C}_c^{\infty}(\Omega)$ such that
$$\|\phi_2-f_n\|_{L^2(\Omega)}+\|\dbarstar\phi_2-\dbarstar f_n\|_{L^2(\Omega)}
\longrightarrow 0 \text{ as } n\to\infty.$$
Therefore $\|\dbarstar f_n\|_{L^2(\Omega)}$ goes to $0$ as 
$n$ goes to $\infty$. However, since $f_n\in\mathcal{C}_c^\infty(\Omega)$ 
integrating by parts twice yields $$\|\dbarstar f_n\|_{L^2(\Omega)}=
\|\nabla f_n\|_{L^2(\Omega)}$$ for all $n\in\mathbb{N}$, i.e., $\nabla f_n$ 
converges to $0$ in $L^2(\Omega)$. Hence $\{f_n\}_{n\in\mathbb{N}}$ is a Cauchy
sequence in $H_0^1(\Omega)$. Therefore, there exists a function 
$f\in H_0^1(\Omega)$ such that $f_n\longrightarrow f$ in $L^2(\Omega)$ and
$\nabla f=0$ almost everywhere. However, using a mollification argument, one can show that  $f$ is constant 
almost everywhere on each connected component of $\Omega$. This is a contradiction 
to Lemma \ref{L:H01notdense} unless 
$f=0$ almost everywhere. Since $f$ and $\phi_2$ are both the $L^2(\Omega)$-limit of 
$f_n$, it follows that $\phi_2=0$ almost everywhere. 
Thus
$$\|\phi\|_{L^2(\Omega)}\leq C\|\dbarstar\phi\|_{L^2(\Omega)}$$
holds. That is, the Poincar{\'e}--Dirichlet inequality holds and 
$\mathfrak{C}(\Omega)\geq\frac{2}{\sqrt{\lambda_{1}(\Omega)}}.$
\end{proof}

\begin{corollary}\label{C:monotonicity}
Let $\Omega'\subset\Omega\subset\mathbb{C}$ be open sets. Then
$\mathfrak{C}(\Omega')\leq
\mathfrak{C}(\Omega)$.
\end{corollary}

\begin{proof}
  If $\mathfrak{C}(\Omega)=\infty$, the claim is trivially true. Otherwise, this follows from Proposition \ref{P:Poincare} and 
  the easily verified fact that 
  $\lambda_1(\Omega)\leq\lambda_1(\Omega')$.
\end{proof}

It also follows from Proposition \ref{P:Poincare}  that $\ker\dbarstar=\{0\}$ which yields a somewhat stronger version of the equivalence of 
(i) and (iv) in Proposition \ref{P:CRpropeq}.

\begin{corollary}
  Let $\Omega\subset\mathbb{C}$ be an open set. Then the following are equivalent.
  \begin{itemize}
    \item[(i)]  $\dbar$ has closed range in $L^{2}(\Omega)$.
    \item[(iv$'$)] There exists a constant $C>0$ such that for all $f\in L^2(\Omega)$    
    there exists a $v\in L^2(\Omega)$
such that $\dbar v=f$ holds in the distributional sense and $$\|v\|_{L^{2}(\Omega)}\leq
C\|f\|_{L^{2}(\Omega)}.$$
  \end{itemize}
\end{corollary}
\begin{proof}
  Only the implication ``(i)$\Rightarrow$(iv')" needs to be proved. Note first that Proposition \ref{P:Poincare} together 
  with the density result
  Lemma \ref{L:Straube} implies that $\ker\dbarstar=\{0\}$. Thus $(\ker\dbarstar)^\perp=L^2(\Omega)$. However,
   $(\ker\dbarstar)^\perp$ equals the closure of the range of $\dbar$, which is closed by assumption. Thus $L^2(\Omega)$ is the range of
   $\dbar$. The implication  ``(i)$\Rightarrow$(iv')" now follows from  ``(i)$\Rightarrow$(iv)" in Proposition \ref{P:CRpropeq}.
\end{proof}

In the following, we observe that the best closed
range constant $\mathfrak{C}(.)$
satisfies a continuity from below property.

\begin{proposition}\label{P:closedness}
Let $\{\Omega_{j}\}_{j\in\mathbb{N}}$ be an increasing sequence
of open sets, set
  $\Omega=\cup_{j\in\mathbb{N}}\Omega_{j}$. If $\dbar$ has 
  closed range in $L^{2}(\Omega_{j})$ with  
constant $C$ for all $j\in \mathbb{N}$, then $\dbar$ has closed
range in $L^{2}(\Omega)$ with
constant $C$. In particular,
$\mathfrak{C}(\Omega)=\lim_{j\to\infty}\mathfrak{C}(\Omega_{j})$.
\end{proposition}

For the proof of Proposition \ref{P:closedness} we shall use the
Bergman projection. Recall that
for
$\Omega\subset\mathbb{C}$, the Bergman projection is the
orthogonal projection
$B^{\Omega}$ of $L^{2}(\Omega)$ onto its closed subspace
$A^{2}(\Omega)$. That is, $B^{\Omega}$
satisfies \begin{itemize}
  \item[(a)] $B^{\Omega}\circ B^{\Omega}=B^{\Omega}$, 
  \item[(b)] $B^{\Omega}h=h$ for all $h\in A^{2}(\Omega)$, 
  \item[(c)] $B^{\Omega}f-f$ is orthogonal to $A^{2}(\Omega)$.
\end{itemize}

\begin{proof}
Let $u\in \dom(\dbar)\cap\left(A^{2}(\Omega)\right)^{\perp}$ be
given. Throughout, write $B$ for
$B^{\Omega}$, $B^{j}$ for $B^{\Omega_{j}}$ and $\|.\|_{\Omega}$
for $\|.\|_{L^{2}(\Omega)}$,
$\|.\|_{\Omega_{j}}$ for $\|.\|_{L^{2}(\Omega_{j})}$. Then for
$j\in\mathbb{N}$,
\begin{align*}
\|u\|_{\Omega}\leq
\|u\|_{\Omega_{j}}+\|u\|_{\Omega\setminus\Omega_{j}}.
\end{align*}
Let $\chi_{j}$ be the characteristic function of $\Omega_{j}$,
and set
$f_{j}=(1-\chi_{j})u^{2}$. Then $f_{j}$ converges to $0$ almost
everywhere in
$\Omega$. Moreover, $|f_{j}|\leq|u|^{2}$ for all
$j\in\mathbb{N}$ and $|u|^{2}$ is in
$L^{1}(\Omega)$. It follows from the dominated convergence
theorem, that
$\lim_{j\to\infty}\int_{\Omega}
f_{j}dV=\int_{\Omega}\lim_{j\to\infty}f_{j}<\infty$, i.e.,
$$\lim_{j\to\infty}\|u\|_{\Omega\setminus\Omega_{j}}=0.$$
Therefore, for a given $\epsilon>0$ there exists a
$j_{0}\in\mathbb{N}$, such that
\begin{align*}
  \|u\|_{\Omega}\leq \|u\|_{\Omega_{j}}+\epsilon
\end{align*}
holds for all $j\geq j_{0}$. Hence
\begin{align*}
  \|u\|_{\Omega}\leq \left\|u-B^{j}u\right\|_{\Omega_{j}}
  +\left\|B^{j}u\right\|_{\Omega_{j}}+\epsilon
\end{align*}
for all $j\geq j_{0}$. Note that $u-B^{j}u$ is orthogonal to
$A^{2}(\Omega_{j})$. Since $u,
B^{j}u\in\dom(\dbar)$, so is $u-B^{j}u$.
As $\dbar$ has closed range in $L^{2}(\Omega_{j})$ with constant
$C$, it now follows
that
\begin{align*}
\left\|u-B^{j}u\right\|_{\Omega_{j}}\leq C\|\dbar
u\|_{\Omega_{j}},
\end{align*}
which yields
\begin{align}\label{I:redo}
\|u\|_{\Omega}\leq C\|\dbar
u\|_{\Omega}+\left\|B^{j}u\right\|_{\Omega_{j}}
   +\epsilon
\end{align}
for all $j\geq j_{0}$.

It remains to estimate the term
$\left\|B^{j}u\right\|_{\Omega_{j}}$. To that end, notice first
that the
 sequence $\{\chi_{j}B^{j}u\}_{j\in\mathbb{N}}$ is uniformly 
bounded in $L^{2}(\Omega)$ by $\|u\|_{L^{2}(\Omega)}$.
Thus it has a weakly convergent subsequence, say, 
$\{\chi_{j_{k}}B^{j_{k}}u\}_{k\in\mathbb{N}}$. That is, there
exists a $g\in L^{2}(\Omega)$ such
that
\begin{align}\label{E:weakconv}
\lim_{k\to\infty}\left(\chi_{j_{k}}B^{j_{k}}u-g,v\right)_{\Omega}=0
  \sjump\;\;\forall\; v\in L^{2}(\Omega).
\end{align}  
We shall first show that $g$ is holomorphic, and then use this
fact to derive that
$\|B^{j_{k}}u\|_{\Omega_{j_{k}}}$ converges to $0$ as $k$ tends
to $\infty$.
It follows from \eqref{E:weakconv} that
\begin{align}\label{E:weakconv2}
\lim_{k\to\infty}\left(\chi_{j_{k}}B^{j_{k}}u-g,\varphi_{z}
\right)_{\Omega}=0\sjump\forall\;\varphi\in\mathcal{C}_{c}^{\infty}(\Omega).
\end{align}
Since $\{\Omega_{j_{k}}\}_{k}$ is an increasing sequence of open
sets, it follows that for each
$\varphi\in\mathcal{C}^{\infty}_{c}(\Omega)$ there exists a
$j_{k_{0}}\in\mathbb{N}$ such that
$\supp\varphi\Subset\Omega_{j_{k}}$ for all $j_{k}\geq
j_{k_{0}}$. Let $S_{\varphi}$ be a smoothly
bounded open set such that
$\supp\varphi\Subset S_{\varphi}\Subset\Omega_{j_{k}}$ for all
$j_{k}\geq
j_{k_{0}}$. Then it follows from integration by parts that
\begin{align*}
  \left(\chi_{j_{k}}B^{j_{k}}u,\varphi_{z} \right)_{\Omega}=
  \left(B^{j_{k}}u,\varphi_{z} \right)_{S_{\varphi}}=
-\left(\left(B^{j_{k}}u\right)_{\bar{z}},\varphi
\right)_{S_{\varphi}}=0
\end{align*}
for all $j_{k}\geq j_{k_{0}}$. This, together with
\eqref{E:weakconv2}, implies that
$(g,\varphi_{z})_{\Omega}=0$
for all $\varphi\in\mathcal{C}_{c}^{\infty}(\Omega)$. Thus,
$g\in A^{2}(\Omega)$, in particular
$(g,u)_{\Omega}=0$.
 The weak convergence \eqref{E:weakconv} yields
\begin{align*}
0=\lim_{k\to\infty}(\chi_{j_{k}}B^{j_{k}}u-g,u)_{\Omega}&=\lim_{k\to\infty}(\chi_{j_{k}}B^{j_{k}}u,u)_{\Omega}\\
 &= \lim_{k\to\infty}(B^{j_{k}}u,u)_{\Omega_{j_{k}}}.
\end{align*}
Since $u-B^{j_{k}}u$ is orthogonal to $A^{2}(\Omega_{j_{k}})$
while
$B^{j_{k}}u\in A^{2}(\Omega_{j_{k}})$, it follows that
\begin{align*}
  0= \lim_{k\to\infty}(B^{j_{k}}u,u)_{\Omega_{j_{k}}}
=\lim_{k\to\infty}\left\|B^{j_{k}}u\right\|_{\Omega_{j_{k}}}^{2}.
\end{align*}
Now repeat all arguments leading up to the estimate
\eqref{I:redo} with $B^{j_{k}}$ instead of
$B^{j}$. Since $\left\|B^{j_{k}}u\right\|_{\Omega_{j_{k}}}$
tends to $0$ as $k\to\infty$, it
follows that
$\|u\|_{\Omega}\leq C\|\dbar u\|_{\Omega}$. 
\medskip

To show that $\mathfrak{C}(\Omega_{j})$ converges to
$\mathfrak{C}(\Omega)$ as $j\to\infty$, set
$C:=\sup\{\mathfrak{C}(\Omega_{j}):j\in\mathbb{N}\}$. By
hypothesis, $C<\infty$. The above
argument then yields
 $\mathfrak{C}(\Omega)\leq C$. However, the 
monotonicity property in Corollary \ref{C:monotonicity}, yields
$\mathfrak{C}(\Omega_{j})\leq\mathfrak{C}(\Omega)$.
Hence $C\leq\mathfrak{C}(\Omega)$. Thus $C=\mathfrak{C}(\Omega)$
holds, which completes the proof.
\end{proof}


\section{Proof of ``(1)$\Rightarrow$(3)'' of Theorem
\ref{T:main}}\label{S:necessity}
\subsection{Special case}\label{SS:specialcase}
The proof of ``(1)$\Rightarrow$(3)'' is done by contraposition,
i.e., we assume that (1) holds
while (3) does not. We shall first consider cases of open sets
for which (3) does not hold in a
particular manner, see hypothesis of Lemma \ref{L:specialcase}
and the example below. The proof of
``(1)$\Rightarrow$(3)'' for these open sets is a straightforward
consequence of Proposition
\ref{P:closedness}.

\begin{lemma}\label{L:specialcase}
Let $\Omega\subset\mathbb{C}$ be an open set. Suppose that for
all $M>0$, there exist a
sequence $\{\delta_{j}\}_{j\in\mathbb{N}}\subset\mathbb{R}^{+}$
with
$\lim_{j\to\infty}\delta_{j}=0$
  and $\{z_{M,j}\}_{j}\subset\mathbb{C}$ such that 
  \begin{itemize}
  \item[(i)] 
$\lcap\left(\Omega^{c}\cap
\mathbb{D}(z_{M,j},M)\right)<\delta_{j}$,
  \item[(ii)] the sequence  of open sets given by 
$\left(\Omega\cap
\mathbb{D}(z_{M,j},M)\right)-z_{M,j}\subset\mathbb{D}(0,M)$ is
increasing in
$j$.
  \end{itemize}
  Then $\dbar$ does not have closed range in $L^{2}(\Omega)$.
\end{lemma}

\begin{example}
For each $(j,\ell) \in \mathbb{Z} \times \mathbb{Z}$, let
$K_{j,\ell}$ be the closed horizontal line segment of
length $\frac{\arctan(j)}{\pi}+\frac{1}{2}$ centered at $j +
\sqrt{-1}\,\ell$.
Let
$$
\Omega = \mathbb{C} \setminus \bigcup_{(j,\ell) \in \mathbb{Z}
\times
\mathbb{Z}} K_{j,\ell}.
$$
Then $\dbar$ does not have closed range in $L^{2}(\Omega)$ by
Lemma~\ref{L:specialcase}.
In fact, for any $m\in\mathbb{N}$, 
$$ 
  \left(\Omega\cap \mathbb{D}(-m,M)\right)+m
  \subset
  \left(\Omega\cap \mathbb{D}(-m-1,M)\right)+(m+1)
  \subset\mathbb{D}(0,M).
$$
Moreover, for any $j\in\mathbb{N}$ there exists an
$m_{j}\in\mathbb{N}$ such that
$$\lcap\left(\Omega^{c}\cap
\mathbb{D}(-m_{j},M)\right)\leq\frac{1}{j}.$$ Hence, conditions
(i)
and (ii) of Lemma~\ref{L:specialcase} are satisfied with
$z_{M,j}=-m_{j}$ and
$\delta_{j}=\frac{1}{j}$.

\begin{figure}[h!t]
\includegraphics{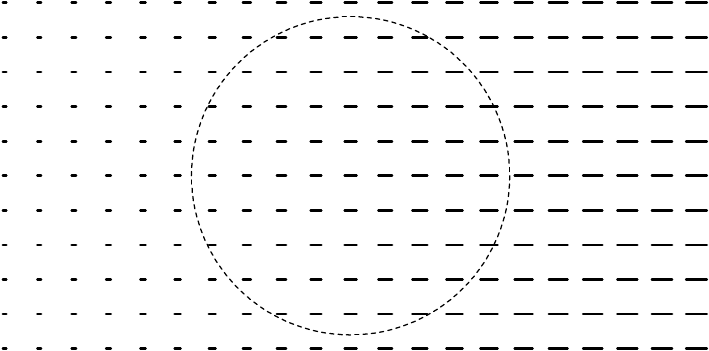}
\caption{A disc in $\Omega$.\label{lines:fig}}
\end{figure}

That $\dbar$ does not have closed range in $L^{2}(\Omega)$ can
also be seen more directly from
Propositions \ref{P:basicprops} and \ref{P:closedness}. To wit,
set $\Omega_m$ to be the shift of
$\Omega$ to the right
by $m$ units. It follows from part (a) of Proposition
\ref{P:basicprops} that
$\mathfrak{C}(\Omega)=\mathfrak{C}(\Omega_{m})$. Further, $
\Omega_{m} \subset \Omega_{m+1}$ for any $m \in \mathbb{Z}$. By
Proposition \ref{P:closedness}, it
now follows that
\begin{align}\label{E:latticeconstant}
  \mathfrak{C}(\Omega)=\mathfrak{C}\left(\Omega_{m}\right)
  =\mathfrak{C}\left(\bigcup_{m\in\mathbb{Z}}\Omega_{m}\right).
\end{align}
However, the complement of $\bigcup_{m\in\mathbb{Z}}\Omega_{m}$
is the lattice
$\mathbb{Z}+\sqrt{-1}\,\mathbb{Z}$, which is a polar set. By
part (c) of Proposition
\ref{P:basicprops}, it follows that
$$\mathfrak{C}(\mathbb{C})=\mathfrak{C}\left(\bigcup_{m\in\mathbb{Z}}\Omega_{m}\right).$$
Since $\mathfrak{C}(\mathbb{C})=\infty$, we obtain from
\eqref{E:latticeconstant} that
$\mathfrak{C}(\Omega)=\infty$.

\end{example}

\begin{proof}[Proof of Lemma \ref{L:specialcase}]
We suppose that $\dbar$ has closed range in $L^{2}(\Omega)$ with constant $C$. Choose $M>0$ such
that
  $\frac{C}{M}<\mathfrak{C}(\mathbb{D}(0,1))$.
  By hypothesis, there exist a sequence of positive scalars 
$\{\delta_{j}\}_{j\in\mathbb{N}}$ with
$\lim_{j\to\infty}\delta_{j}=0$
  and
  $\{z_{j}\}_{j\in\mathbb{N}}\in\mathbb{C}$ such that  
$$\lcap\left(\Omega^{c}\cap\mathbb{D}(z_{j},M)\right)<
\delta_{j}.$$

  Define $D_{j}$ to be the set obtained from translating 
$\Omega\cap\mathbb{D}(z_{j},M)$ by $-z_{j}$ and then scaling it
by a factor of $\frac{1}{M}$,
i.e.,
 $$D_{j}=\left\{z\in\mathbb{D}(0,1):Mz+z_{j}
\in\Omega\right\}.$$ Then by properties (a),(b) and (c) of
Proposition \ref{P:basicprops},
 \begin{align}\label{E:constantj}
   \mathfrak{C}(D_{j})\leq\frac{C}{M}.
 \end{align}
As the logarithmic capacity satisfies analogous properties, see
part (c) of
Theorem 5.1.2 in \cite{Rans95}, we also have
$\lcap(\mathbb{D}(0,1)\setminus
D_{j})\leq\frac{\delta_{j}}{M}$.
Since $D_{j}\subset D_{j+1}$, it follows from monotonicity, see
Theorem 5.1.2 (a) in
\cite{Rans95}, that
$$\lcap\left(\mathbb{D}(0,1)\setminus\bigcup_{j=1}^{\infty}D_{j}
\right)\leq
   \lcap\left(\mathbb{D}(0,1)\setminus D_{k} \right)
\leq\frac{\delta_{k}}{M}\sjump\;\;\forall\;k\in\mathbb{N}. $$ By part (c) of Proposition
\ref{P:basicprops} we then get that
$$\mathfrak{C}(\bigcup_{j=1}^{\infty}D_{j})=\mathfrak{C}(\mathbb{D}(0,1)).$$
Moreover, Proposition \ref{P:closedness} yields for any given
$\epsilon>0$
 a $j\in\mathbb{N}$ such that
$$\mathfrak{C}(D_{j})>\mathfrak{C}\left(\bigcup_{j=1}^{\infty}
D_{j}\right)-\epsilon.$$
 By \eqref{E:constantj}, it then follows that
$$\frac{C}{M}>\mathfrak{C}\left(\mathbb{D}(0,1)\right)-\epsilon.$$
This is a contradiction to the choice of $M$ for $\epsilon>0$
sufficiently small.
\end{proof}

\medskip

\subsection{General case}

\begin{proposition}\label{P:necessity}
Let $\Omega\subset\mathbb{C}$ be an open set. Suppose $\dbar$
has closed range in $L^{2}(\Omega)$.
Then $\rho_{\lcap}(\Omega)<\infty$, i.e.,
there exist positive constants $M$ and $\delta$ such that for
each $z\in\mathbb{C}$ there exists a
compact set
$K\subset\Omega^{c}$ such that
$$\lcap\left(K\cap\mathbb{D}(z,M)\right)\geq\delta.$$
\end{proposition}

\begin{lemma}\label{L:pwsmooth}
Let $K\subset\mathbb{C}$ be a compact set and $\epsilon>0$.
\begin{itemize}
  \item[(i)] There exists a compact set  
$K_{\epsilon}\subset\mathbb{C}$ such that $\mathbb{C}\setminus
K_{\epsilon}$ has smooth boundary,
$K\subset K_{\epsilon}$, and
$\lcap(K_{\epsilon})\leq\lcap(K)+\epsilon$.
\item[(ii)] Suppose that $K$ has smooth boundary and
$K\cap\mathbb{D}(0,1)$ is non-empty. Then
there exists a compact set
$K_{\epsilon}\subset\overline{\mathbb{D}(0,1)}$ such that
$K\cap\overline{\mathbb{D}(0,1)} \subset K_{\epsilon}$,
$\lcap(K_{\epsilon})\leq\lcap(K)+\epsilon$
and $\mathbb{D}(0,1)\setminus K_{\epsilon}$ is a smoothly
bounded set.
\end{itemize}  
  \end{lemma}

\begin{proof}
(i) Let $K\subset\mathbb{C}$ be a compact set, $\epsilon>0$ be
given. Then, by a result due to
Choquet, there exists an open, bounded set $U_{\epsilon}$ such
that $K\subset U_{\epsilon}$ and
$\lcap(U_{\epsilon})\leq \lcap(K)+\epsilon$. Moreover, there
exists a smoothly bounded, open set
$U'_{\epsilon}$ such that $K\subset U'_{\epsilon}\subset
U_{\epsilon}$. Set
$K_{\epsilon}=\overline{U'}_{\epsilon}$. It then follows from
monotonicity of the logarithmic
capacity that $\lcap(K_{\epsilon})\leq\lcap(K)+\epsilon$.
  
(ii) Let $K\subset\mathbb{C}$ be a compact set with smooth
boundary such that
$K\cap\mathbb{D}(0,1)$ is non-empty. Let $\epsilon>0$ be given.
As in part (i) there exists an
open set $U_{\epsilon}$ such that $K\subset U_{\epsilon}$ and
$\lcap(U_{\epsilon})\leq\lcap(K)+\epsilon$.
 Then there exist an open set $U'_{\epsilon}$ such that 
$$K\cap\overline{\mathbb{D}(0,1)}\subset
U'_{\epsilon}\cap\overline{\mathbb{D}(0,1)}
 \subset U_{\epsilon}\cap\overline{\mathbb{D}(0,1)}$$
and $\mathbb{D}(0,1)\setminus \overline{U'_{\epsilon}}$ has
smooth boundary.
 With $K_{\epsilon}:=\overline{U'_{\epsilon}}$, (ii) follows.
\end{proof}

\medskip

\begin{proof}[Proof of Proposition \ref{P:necessity}]
The proof is done by contraposition, i.e., we assume that
$\dbar$ has closed range in
$L^{2}(\Omega)$ with
constant $C$ while $\rho_{\lcap}(\Omega)=\infty$. That is, we
assume that for
each $M, \delta>0$, there is a point $z_{M,\delta}\in\mathbb{C}$
such that for any compact set
 $K\subset\Omega^{c}$
$$\lcap\left(K\cap\mathbb{D}(z_{M,\delta},M)\right)<\delta. $$
Choose an $M>0$ such that 
\begin{align}\label{E:choiceofM}
  \frac{4M^{2}}{C^{2}}>\lambda_{1}(\mathbb{D}(0,1)).
\end{align}
Let $\{\delta_{j}\}_{j\in\mathbb{N}}$ be a sequence of positive
scalars with
$\lim_{j\to\infty}\delta_{j}=0$. For
$N>M$ and positive $\epsilon_{j}<\delta_{j}$ we may choose
$z_{N,\epsilon_{j}}$ such that any
compact set
contained in $ \Omega^{c}\cap\mathbb{D}(z_{N,\epsilon_{j}},N)$
has logarithmic capacity less than
$\epsilon_{j}$. For $j\in\mathbb{N}$, set
$$K_{j}^{1}=\Omega^{c}\cap\overline{\mathbb{D}(z_{N,\epsilon_{j}},M)}.$$
Then
$K_{j}^{1}\subset\overline{\mathbb{D}(z_{N,\epsilon_{j}},M)}$ is
compact and, by inner regularity
of the logarithmic capacity, see \cite[Theorem
5.1.2(b)]{Rans95},
$\lcap(K_{j}^{1})\leq\epsilon_{j}$ for all $j\in\mathbb{N}$.
By Lemma \ref{L:pwsmooth}, for any $j\in\mathbb{N}$ there exists
a compact set $K_{j}^{2}$ such
that
$\mathbb{D}(z_{N,\epsilon_{j}},M)\setminus K_{j}^{2}$ has smooth
boundary,
$K_{j}^{1}\subset K_{j}^{2}$ and $\lcap(K_{j}^{2})<\delta_{j}$. Now set 
 \begin{align*}
K_{j}&=\frac{1}{M}\left(K_{j}^{2}\cap\overline{\mathbb{D}(z_{N,\epsilon_{j}},M)}-z_{N,\epsilon_{j}}\right)\\
&=\left\{z\in\overline{\mathbb{D}(0,1)}:
Mz+z_{N,\epsilon_{j}}\in K_{j}^{2}
    \cap\overline{\mathbb{D}(z_{N,\epsilon_{j}},M)}
    \right\}.
 \end{align*}
Then $K_{j}\subset\overline{\mathbb{D}(0,1)}$ is compact such
that
 $\mathbb{D}(0,1)\setminus K_{j}$ has smooth boundary and 
$\lcap(K_{j})<\frac{\delta_{j}}{M}$ for $j\in\mathbb{N}$. Next, set
$D_{j}=\mathbb{D}(0,1)\setminus K_{j}$. Then $D_{j}$ is
open, has smooth boundary, and
$\mathfrak{C}(D_{j})\leq \frac{C}{M}$ for all $j\in\mathbb{N}$,
hence
 \begin{align}\label{E:eigenvalueestimate}
\lambda_{1}(D_{j})\geq\frac{4M^{2}}{C^{2}}>\lambda_{1}(\mathbb{D}(0,1))+\epsilon
 \end{align}
for some $\epsilon>0$ by \eqref{E:choiceofM}. In the following,
we will show that
 \begin{align}\label{E:eigenvaluelimit}
\lim_{j\to\infty}\lambda_{1}(D_{j})=\lambda_{1}(\mathbb{D}(0,1)). \end{align}
This would conclude the proof as \eqref{E:eigenvaluelimit} is a contradiction to
\eqref{E:eigenvalueestimate}.
 
Let $\varphi$ be an eigenfunction corresponding to the first
eigenvalue for the Dirichlet problem
in
$\mathbb{D}(0,1)$. Then
$\varphi\in\mathcal{C}^{\infty}(\overline{\mathbb{D}(0,1)})$,
see e.g.,
the remark following Theorem~1 in \cite[Section 6.5]{Evans98}.
Moreover, we may assume that
$0\leq\varphi\leq 1$ on $\mathbb{D}(0,1)$.
Next, let $h_{j}$ be the harmonic function in $D_{j}$ such that
$h_{j}=\varphi$ on $bD_{j}$. Since
$bD_{j}$
is smooth and the boundary data $\varphi$ is smooth up to the
boundary of $\mathbb{D}(0,1)$,
it follows that $h_{j}$ is smooth up to the boundary of $D_{j}$
as well.
 
 Set $\psi_{j}(z)=\varphi(z)-h_{j}(z)$ 
for $z\in\overline{D}_{j}$.Then
$\psi_{j}\in\mathcal{C}^{\infty}(\overline{D_{j}})$,
$\psi_{j}=0$
on $bD_{j}$ and
 $\triangle\psi_{j}=\triangle\varphi$ on $D_{j}$.
 By \eqref{E:RayleighLaplace}
 \begin{align*}
   0<\lambda_{1}(D_{j})
\leq
\frac{\|\triangle\psi_{j}\|_{L^{2}(D_{j})}}{\|\psi_{j}\|_{L^{2}(D_{j})}}.
 \end{align*}
It then follows from monotonicity of the first Dirichlet
eigenvalue, that
 \begin{align*}
   \frac{1}{\lambda_{1}(\mathbb{D}(0,1))}\geq
    \frac{1}{\lambda_{1}(D_{j})}
&\geq\frac{\|\psi_{j}\|_{L^{2}(D_{j})}}{\|\triangle\psi_{j}\|_{L^{2}(D_{j})}}\\
&=\frac{\|\varphi-h_{j}\|_{L^{2}(D_{j})}}{\|\triangle\varphi\|_{L^{2}(D_{j})}}\\
    &\geq
\frac{\|\varphi\|_{L^{2}(D_{j})}}{\lambda_{1}(\mathbb{D}(0,1))\cdot\|\varphi\|_{L^{2}(D_{j})}}-\frac{\|h_{j}\|_{L^{2}(D_{j})}}{\lambda_{1}(\mathbb{D}(0,1))\cdot\|\varphi\|_{L^{2}(D_{j})}}.
\end{align*}
As $\lcap(\mathbb{D}(0,1)\setminus D_{j})=\lcap(K_{j})$ goes to
zero as $j\to\infty$, it follows
that
$\|\varphi\|_{L^{2}(D_{j})}$ approaches
$\|\varphi\|_{L^{2}(D)}$. That is,
 \begin{align*}
1\geq
\frac{\lambda_{1}(\mathbb{D}(0,1))}{\lambda_{1}(D_{j})}\geq
1-\frac{\|h_{j}\|_{L^{2}(D_{j})}}{\|\varphi\|_{L^{2}(D_{j})}}. \end{align*}
To prove that \eqref{E:eigenvaluelimit} holds, it remains to
show that
 $\lim_{j\to\infty}\|h_{j}\|_{L^{2}(D_{j})}=0$.
Note first that the maximum principle yields
$0<h_{j}\leq 1$ on $D_{j}$, hence
$|h_{j}|^{2}\leq h_{j}$.
Moreover, if $g_{j}\in\mathcal{C}(\overline{D}_{j})$ is a
positive, harmonic function on $D_{j}$
such that $g_{j}=1$ on
$bK_{j}\setminus b\mathbb{D}(0,1)$ and $g_{j}\geq 0$ on
$b\mathbb{D}(0,1)\cap bD_{j}$, then
 $0<h_{j}\leq g_{j}$ on $D_{j}$. In particular,
$$\int_{D_{j}}|h_{j}|^{2}\;dA\leq
\int_{D_{j}}h_{j}\;dA\leq\int_{D_{j}}g_{j}\;dA.$$
So it suffices to show that there is such a sequence
$\{g_{j}\}_{j\in\mathbb{N}}$ whose
$L^{1}(D_{j})$-integral converges to $0$.
 
To construct such $g_{j}$, let $\nu_{j}$ be the equilibrium
measure for $K_{j}$, and set
$$J(\nu_{j})=\int_{\mathbb{C}}\int_{\mathbb{C}}\ln\left(\frac{2}{|z-w|}
\right)\;d\nu_{j}(w)\;d\nu_{j}(z).$$
Note that $J(\nu_{j})=\ln(2)-I(\nu_{j})$. Furthermore, as
$\lim_{j\to\infty}\lcap(K_{j})=0$, it
follows that
$\lim_{j\to\infty}I(\nu_{j})=-\infty$, hence
$\lim_{j\to\infty}J(\nu_{j})=\infty$. Hence for $j$
sufficiently large we may define
$$g_{j}(z)=\frac{1}{J(\nu_{j})}\int_{\mathbb{C}}\ln\left(\frac{2}{|z-w|}
\right)\;d\nu_{j}(w).$$
 
Then $g_{j}$ is a positive, harmonic function on $D_{j}$, which
is non-negative on
 $b\mathbb{D}(0,1)\cap bD_{j}$. We claim that  $g_{j}$
equals $1$ on $bK_{j}\setminus b\mathbb{D}(0,1)$ and is
continuous on $\overline{D_{j}}$. To show
the
former, we first note that any boundary point of $D_{j}$ is a
regular boundary point since any
smooth
defining function of $D_{j}$ serves as a subharmonic barrier
function, see \cite[Def.
4.1.4]{Rans95}. This
implies that the potential function $p_{j}$ associated to the
equilibrium measure $\nu_{j}$ of
$K_{j}$ is equal
to $I(\nu_{j})$ on $bK_{j}\setminus\mathbb{D}(0,1)$, see
\cite[Theorem 4.2.4]{Rans95}. However,
this implies
 that $g_{j}=1$ on
$bK_{j}\setminus b\mathbb{D}(0,1)$. It also implies that
$p_{j}\in\mathcal{C}(\overline{D}_{j})$,
see
\cite[Theorem 3.1.3]{Rans95}. Therefore,
$g_{j}\in\mathcal{C}(\overline{D}_{j})$. It remains to be
shown that
 $\int_{D_{j}}g_{j}\;dA$ converges to $0$ as $j\to\infty$.
 
 We compute
 \begin{align*}
   \int_{D_{j}}g_{j}(z)\;dA(z)
   &=\frac{1}{J(\nu_{j})}
\int_{D_{j}}\int_{\mathbb{C}}\ln\left(\frac{2}{|z-w|}\right)\;d\nu_{j}(w)\;dA(z)\\
&=\frac{1}{J(\nu_{j})}\int_{\mathbb{C}}\int_{D_{j}}\ln\left(\frac{2}{|z-w|}\right)\;dA(z)\;d\nu_{j}(w).
  \end{align*} 
  Note that
  \begin{align*}
    \int_{D_{j}}\ln\left(\frac{2}{|z-w|}\right)\;dA(z)
&\leq
\int_{\mathbb{D}(0,1)}\ln\left(\frac{2}{|z-w|}\right)\;dA(z)\\
&\leq\int_{\mathbb{D}(w,2)}\ln\left(\frac{2}{|z-w|}\right)\;dA(z)\\
&=\int_{\mathbb{D}(0,2)}\ln\left(\frac{2}{|z|}\right)\;dA(z)=2\pi.
  \end{align*}
Therefore $\lim_{j\to\infty}\int_{D_{j}}g_{j}(z)\;dA(z)=0$,
i.e.,
$\lim_{j\to\infty}\|h_{j}\|_{L^{2}(D_{j})}=0$, and hence
\eqref{E:eigenvaluelimit} holds, which
concludes the proof.
\end{proof}


\section{Proof of ``(3)$\Rightarrow$(4)'' of Theorem
\ref{T:main}}\label{S:sufficiency}

\begin{proposition}
Let $\Omega\subset\mathbb{C}$ be an open set. Suppose
$\rho_{\lcap}(\Omega)<\infty$, i.e., there
exist positive constants $M$ and $\delta$ such that for each
$z\in\mathbb{C}$ there exists a
compact set $K\subset\Omega^{c}$ such that
$$\lcap\left(K\cap\mathbb{D}(z,M)\right)\geq\delta.$$
Then there exists a bounded function
$\varphi\in\mathcal{C}^{\infty}(\Omega)$ and a positive
constant $c$ such that $\varphi_{z\bar{z}}(z)\geq c$ for all
$z\in\Omega$.
\end{proposition}

\begin{proof}
We first observe that whenever there exists a compact, non-polar
set
$K\subset\Omega^{c}$, then $\Omega$ admits a non-constant,
bounded, real-analytic, subharmonic
function. In fact, let $\nu_{K}$ be the equilibrium measure of
$K$ such that $\supp(\nu_{K})\subset K$. 
The associated potential $p_{K}$ is given by
$$p_{K}(z)=\int_{\mathbb{C}}\ln|z-w|\;dv_{K}(w). $$
By Frostman's Theorem, $p_{K}(z)\geq\ln\left(\lcap(K)\right)$
for any $z\in\Omega$. Thus the
values of $e^{-p_{K}(z)}$ are in $(0,1/\lcap(K)]$. Moreover,
$p_{K}$ is harmonic, hence
real-analytic. As $p_{K}$ is also non-constant, it follows that
$e^{-p_{K}}$ is a non-constant,
bounded, real-analytic, subharmonic function on $\Omega$.

These kinds of functions will be the building blocks for 
the construction of $\varphi$. In 
fact, we will show that there exists a sequence
$\{K_{j}\}_{j\in\mathbb{N}}$ in
$\Omega^{c}$ and constants $c_{1}$, $c_{2}>0$ such that
\begin{itemize}
\item[(i)] $\varphi(z):=\sum_{j\in\mathbb{N}}e^{-4p_{K_{j}}(z)}$
is a
  smooth function on $\Omega$,
  \item[(ii)] $0\leq\varphi(z)\leq c_{1}$ for all $z\in\Omega$,
\item[(iii)] $\triangle\varphi(z)\geq c_{2}$ for all
$z\in\Omega$.
\end{itemize}

\medskip

Without loss of generality, we may assume that $2\delta<M$.
We claim that for each $(j,k)\in\mathbb{Z}\times\mathbb{Z}$, we
may choose a compact set $K_{j,k}$
such that $\lcap(K_{j,k})\geq \delta$ and
$$K_{j,k}\subset \mathbb{D}\left((2jM,2kM),M+2\delta
\right)\cap\Omega^{c}.$$ This can be seen as
follows. Suppose that for a given $(j,k)$ there was no such
compact set. If there exists a
$$z\in\Omega\cap\mathbb{D}\left((2jM,2kM),\delta\right),$$ then,
by hypothesis, the logarithmic
capacity of
$\overline{\mathbb{D}(z,M)}\cap\Omega^{c}$ is at least $\delta$.
This is a
contradiction to our assumption since
$$\overline{\mathbb{D}(z,M)}\cap\Omega^{c}\subset\mathbb{D}\left((2jM,2kM),M+2\delta
\right)\cap\Omega^{c}.$$ See Figure \ref{IIimpliesIII:fig}. Thus
$\Omega^{c}$ contains
$\overline{\mathbb{D}\left((2jM,2kM),\delta\right)}$ which is a
compact set of logarithmic
capacity $\delta$. This proves the claim.
\begin{figure}[h!t]
\includegraphics{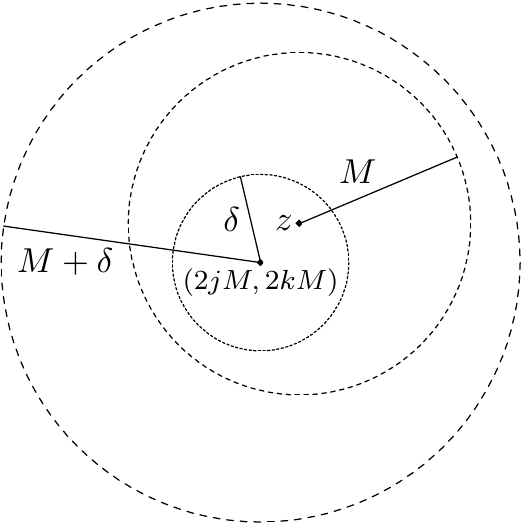}
\caption{The disc of radius $M+\delta$ contains the disc
of radius $M$ centered at $z$.\label{IIimpliesIII:fig}}
\end{figure}

For each $(j,k)\in\mathbb{Z}\times\mathbb{Z}$, choose a compact
set $K_{j,k}$ as described above.
Let $p_{K_{j,k}}$ be the associated potential; for the sake of
brevity, write $p_{j,k}$ in place
of $p_{K_{j,k}}$.
We shall show that the series 
$$\sum_{(j,k)\in\mathbb{Z}\times\mathbb{Z}}e^{-4p_{j,k}(z)}$$
converges for any $z\in\Omega$. To do this, we will fix a
$z\in\Omega$ and show convergence of the
series for a particular enumeration of
$\mathbb{Z}\times\mathbb{Z}$. As the terms of the series
are non-negative, it will then follow that the series converges
(to the same value) for any choice
of enumeration.

For given $z\in\mathbb{C}$, write $\mathcal{Q}(z,L)$ for the
closed square with center $z$
and side length $2L$.
For fixed $z\in\Omega$, let
$(j_{0},k_{0})\in\mathbb{Z}\times\mathbb{Z}$ such that
$z\in\mathcal{Q}\left((2j_{0}M,2k_{0}M),M\right)$. For
$\lambda\in\mathbb{N}$, set
$$A_{\lambda}=\left\{(j,k)\in\mathbb{Z}\times\mathbb{Z}:
\max\{|j-j_{0}|,|k-k_{0}|\right\}=\lambda\}.$$
A straightforward computation yields that
$\card(A_{\lambda})=8\lambda$ for
$\lambda\geq 1$. Next note that
$$p_{j,k}(z)\geq\ln\delta\;\;\sjump\forall\;(j,k)\in A_{1}.$$
Furthermore, if
$(j,k)\in A_{\lambda}$ for some integer $\lambda\geq 2$ and
$w\in K_{j,k}$,
then $|z-w|\geq \left(\lambda-1\right)M$. Hence
$$p_{j,k}(z)\geq \ln\left(\left(\lambda-1\right)M\right)
\int_{\mathbb{C}}1\;dv_{j,k}(w)=\ln\left(\left(\lambda-1\right)M\right).$$
Therefore,
\begin{align*}
  \sum_{\lambda=1}^{\infty}\sum_{(j,k)\in A_{\lambda}}
  e^{-4p_{j,k}(z)}
&\leq\card(A_{1})\delta^{-4}+\sum_{\lambda=2}^{\infty}\card(A_{\lambda})
  \frac{1}{M^{4}(\lambda-1)^{4}}\\
  &=8 \delta^{-4}+\frac{8}{M^{4}}\sum_{\lambda=2}^{\infty}
  \frac{\lambda}{(\lambda-1)^{4}}.
\end{align*}
Hence, by the Weierstra{\ss} $M$-test, the series is uniformly
convergent on $\Omega$. In
particular, $\varphi(z):=\sum_{(j,k)\in\mathbb{Z}\times\mathbb{Z}}e^{-4p_{j,k}(z)}$
is a well-defined function
which is continuous and bounded on $\Omega$.

To see that $\varphi$ is in fact in
$\mathcal{C}^{\infty}(\Omega)$, it suffices to note that
all derivatives of the $p_{j,k}$'s are locally, uniformly
bounded.
For instance, after computing
$$
\frac{\partial }{\partial
z}\left(p_{j,k}(z)\right)=\frac{1}{2}\int_{\mathbb{C}}
\frac{1}{z-w}\;dv_{j,k}(w),
$$
we note that for any $z_{0}\in\Omega$, there exists a
neighborhood $U_{z_{0}}\subset\Omega$ of $z_{0}$
and a constant $c_{z_{0}}>0$, such that
$$|z-w|>c\sjump\;\;\forall\;z\in
U_{z_{0}},\;\forall\;w\in\Omega^{c}.$$
Hence for $z\in U_{z_{0}}$, it follows that
\begin{align*}
\Bigl|\sum_{(j,k)\in\mathbb{Z}\times\mathbb{Z}}\frac{\partial}{\partial
z}
  \left(e^{-4p_{j,k}(z)} \right)\Bigr|&
  =4\Bigl|\sum_{(j,k)\in\mathbb{Z}\times\mathbb{Z}}
e^{-4p_{j,k}(z)}\frac{\partial}{\partial
z}\left(p_{j,k}(z)\right)\Bigr|\\
  &\leq
2c_{z_{0}}\sum_{(j,k)\in\mathbb{Z}\times\mathbb{Z}}e^{-4p_{j,k}(z)}
  =2c_{z_{0}}\varphi(z),
\end{align*}  
i.e., the series 
$\sum_{(j,k)\in\mathbb{Z}\times\mathbb{Z}}\frac{\partial}{\partial
z}
\left(e^{-4p_{j,k}(z)} \right)$ 
converges locally uniformly. Thus it converges to
$\frac{\partial\varphi}{\partial z}$.
Similarly, if $D$ represents any differential operator, the
series
$$\sum_{(j,k)\in\mathbb{Z}\times\mathbb{Z}} D\left(e^{-4
p_{j,k}(z)}\right)$$
is convergent locally uniformly. In particular, the series
converges to $D\varphi$, and
$\varphi\in\mathcal{C}^{\infty}(\Omega)$.

The last observation yields in particular
$$\triangle\varphi(z)=16\sum_{(j,k)\in\mathbb{Z}\times\mathbb{Z}}e^{-4p_{j,k}(z)}
\left|\nabla
p_{j,k}(z)\right|^{2}\;\;\sjump\forall\;z\in\Omega,$$ i.e.,
$\varphi$ is subharmonic on $\Omega$. It remains to be shown
that $\varphi$ satisfies condition
(iii). For that, it suffices to show that there exists a
constant $c>0$ such that for each
$z\in\Omega$ there exists a $(j,k)\in\mathbb{Z}\times\mathbb{Z}$
such that
$$e^{-p_{j,k}(z)}\left|\nabla p_{j,k}(z) \right|^{2}\geq c.$$
A straightforward computation yields 
\begin{align*}
  \left|\nabla p_{j,k}(z)\right|&=\frac{1}{2}
\left|\int_{\mathbb{C}}\frac{\re(\bar{z}-\bar{w})+i\im(\bar{z}-\bar{w})}{|z-w|^{2}}
  dv_{j,k}(w)\right|\\
  &\geq 
\frac{1}{2}\left|\int_{\mathbb{C}}\frac{\re(z-w)}{|z-w|^{2}}  dv_{j,k}(w) \right|
\end{align*}
for any given $z\in\Omega$ and
$(j,k)\in\mathbb{Z}\times\mathbb{Z}$. Now let $z\in\Omega$ be
given. Then, as before, $z\in\mathcal{Q}((2j_{0}M,2k_{0}M),M)$
for some
$(j_{0},k_{0})\in\mathbb{Z}\times\mathbb{Z}$. Let
$(j_{1},k_{1})=(j_{0}-2,k_{0}-2)$. Note that for
any $w\in K_{j_{1},k_{1}}$
$$
\re(z-w)=\re(z)-\re(w)\geq
M\;\;\text{and}\;\;\sqrt{2}M<|z-w|<\sqrt{98}M
$$
holds. Thus $|\nabla p_{j_{1},k_{1}}(z)|\geq\frac{1}{2\cdot
98M}$.
Moreover,  $e^{-4p_{j_{1},k_{1}}(z)}\geq \frac{1}{2}M^{-4}$. 
Hence, $\triangle\varphi(z)\geq \frac{2}{49 M^{5}}$ for all
$z\in\Omega$.
See Figure \ref{IIimpliesIII-2:fig}.
\begin{figure}[h!t]
\includegraphics{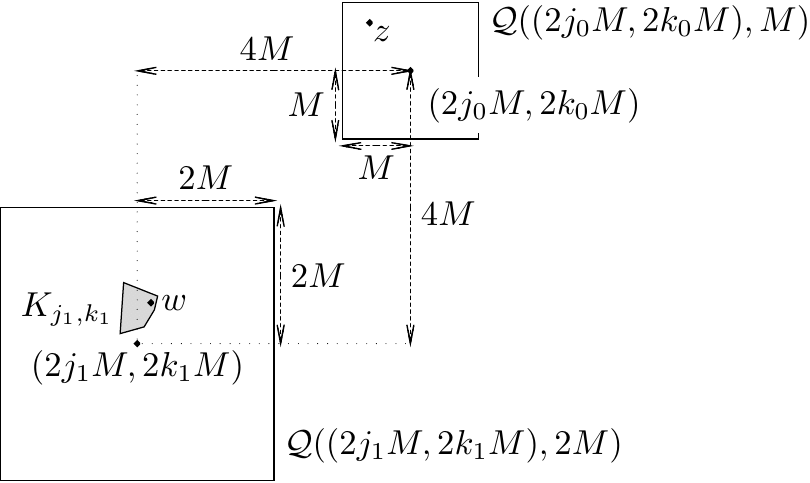}
\caption{The squares at $(2j_0M,2k_0M)$ and
$(2j_1M,2k_1M)$.\label{IIimpliesIII-2:fig}}
\end{figure}
\end{proof}


\section{Proof of ``(3)$\Rightarrow$(4)'' of Theorem
\ref{T:Bergmanspace}}\label{S:Bergmanspace}
\begin{proposition}
Let $\Omega\subset\mathbb{C}$ be an open set such that
$\lcap(\Omega^{c})>0$. Then there exists a bounded
$\varphi\in\mathcal{C}^{\infty}(\Omega)$ such that
$\triangle\varphi>0$ on $\Omega$.
\end{proposition}
\begin{proof}
Note that there exists a compact set $K\subset\Omega^{c}$ such
that $\lcap(K)>0$. Let $\nu$ be the
equilibrium measure of $K$, and $p$ be the associated potential
function, i.e.,
  $$p(z)=\int_{C}\ln|z-w|\;d\nu(w).$$
Recall that $p$ is harmonic in $\Omega$ and
$p\geq\ln(\lcap(K))$. Hence $\varphi:=e^{-p}$ is
smooth and
bounded on $\Omega$. Moreover, $\triangle\varphi=e^{-p}|\nabla
p|^{2}$.
  
We shall show first that there is an open ball $B$ such that
$K\subset B$ and
$|\frac{\partial p}{\partial z}(z)|>0$ for all $z\in
B^{c}\cap\Omega$.
Let $B$, $B'$ be concentric open balls, such that $K\subset B'$,
and the radius of $B$ is
  twice the radius of $B'$.
  Let $z\in\Omega\cap B^{c}$ be given, 
  $\theta\in[0,2\pi)$ fixed and chosen later.
  See Figure \ref{angle:fig}.
  Then
  \begin{align*}
    \left|\frac{\partial p}{\partial z}(z) \right|
&=\frac{1}{2}\left|\int_{\mathbb{C}}\frac{\bar{z}-\bar{w}}{|z-w|^{2}}e^{-i\theta}
\;d\nu(w)\right|\\
&\geq\frac{1}{2}\left|\int_{\mathbb{C}}\frac{\re\left((z-w)e^{-i\theta}
\right)}{|z-w|^{2}}\;d\nu(w) \right|\\
&=\frac{1}{2}\left|\int_{\mathbb{C}}\frac{\cos(\alpha(z-w)-\theta)}{|z-w|}\;d\nu(w)
\right|,
  \end{align*}
where $\alpha(z-w)$ is the branch of the argument of $z-w$ in
$[0,2\pi)$. Since the radius of $B$
is
twice the radius of $B'$, it follows that there is some
$c\in(0,\pi/2)$ such that
$$\left|\alpha(z-w)-\alpha(z-w^{*})
\right|<c\;\;\sjump\forall\;w, w^{*}\in K.$$
This allows us to choose $\theta$ such that
$\alpha(z-w)-\theta\in[0,c)\subset[0,\pi)$ for all
$w\in K$. Hence
$\cos(\alpha(z-w)-\theta)>0$ for all $w\in K$. Thus
$|\frac{\partial p}{\partial z}(z)|>0$ for all $z\in\Omega\cap
B^{c}$.
\begin{figure}[h!t]
\includegraphics{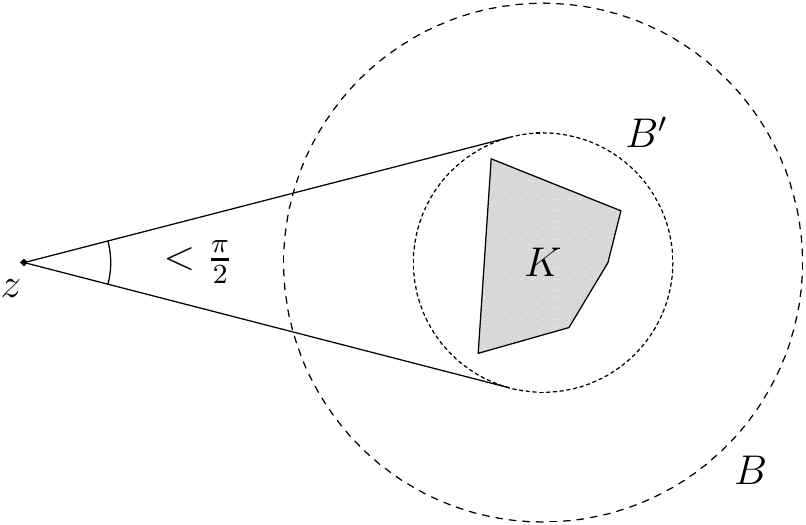}
\caption{From $z$ outside $B$, the disc $B'$, and hence $K$,
subtends an angle less than
$\frac{\pi}{2}$.\label{angle:fig}}
\end{figure}
  
Let $B''$ be an open set containing $\overline{B}$. Then, since
$p$ is smooth in
  $\Omega$ and $|\nabla p|>0$ on $\Omega\cap B^{c}$ 
there exists a constant $C>0$ such that $\triangle\varphi>C$
on$\overline{\Omega}\cap\left(\overline{B''}\setminus B\right)$.
Let $z_{0}$ be the center of $B'$,
let
$\chi\in\mathcal{C}^{\infty}_{c}(\mathbb{R})$ such that
$\chi(|z-z_{0}|^{2})=|z-z_{0}|^{2}$ for
all $z\in B$ and
  $\text{supp}(\chi(|z-z_{0}|^{2}))\subset B''$.
Then $\varphi+\epsilon\chi(|z-z_{0}|^{2})$ is a smooth, bounded
function on $\Omega$ which is
strictly
subharmonic everywhere as long as $\epsilon>0$ is sufficiently
small.
\end{proof}

\bibliographystyle{acm}
\bibliography{References}

\begin{thebibliography}{10}

\bibitem{Car67}
{\sc Carleson, L.}
\newblock {\em Selected problems on exceptional sets}.
\newblock Van Nostrand Mathematical Studies, No. 13. D. Van Nostrand Co., Inc.,
  Princeton, N.J.-Toronto, Ont.-London, 1967.

\bibitem{Con95}
{\sc Conway, J.~B.}
\newblock {\em Functions of one complex variable. {II}}, vol.~159 of {\em
  Graduate Texts in Mathematics}.
\newblock Springer-Verlag, New York, 1995.

\bibitem{Evans98}
{\sc Evans, L.~C.}
\newblock {\em Partial differential equations}, vol.~19 of {\em Graduate
  Studies in Mathematics}.
\newblock American Mathematical Society, 1998.

\bibitem{GaHaHe17}
{\sc Gallagher, A.-K., Harz, T., and Herbort, G.}
\newblock On the dimension of the {B}ergman space for some unbounded domains.
\newblock {\em J. Geom. Anal. 27}, 2 (2017), 1435--1444.

\bibitem{HerMcN16}
{\sc Herbig, A.-K., and McNeal, J.~D.}
\newblock On closed range for {$\overline\partial$}.
\newblock {\em Complex Var. Elliptic Equ. 61}, 8 (2016), 1073--1089.

\bibitem{Hormander65}
{\sc H{\"o}rmander, L.}
\newblock {$L^{2}$} estimates and existence theorems for the {$\bar \partial
  $}\ operator.
\newblock {\em Acta Math. 113\/} (1965), 89--152.

\bibitem{Rans95}
{\sc Ransford, T.}
\newblock {\em Potential theory in the complex plane}, vol.~28 of {\em London
  Mathematical Society Student Texts}.
\newblock Cambridge University Press, Cambridge, 1995.

\bibitem{Sou99}
{\sc Souplet, P.}
\newblock Geometry of unbounded domains, {P}oincar\'{e} inequalities and
  stability in semilinear parabolic equations.
\newblock {\em Comm. Partial Differential Equations 24}, 5-6 (1999), 951--973.

\bibitem{Sou00}
{\sc Souplet, P.}
\newblock Decay of heat semigroups in {$L^\infty$} and applications to
  nonlinear parabolic problems in unbounded domains.
\newblock {\em J. Funct. Anal. 173}, 2 (2000), 343--360.

\bibitem{Straube10}
{\sc Straube, E.~J.}
\newblock {\em Lectures on the {$L^2$}-{S}obolev theory of the
  {$\overline{\partial}$}-{N}eumann problem}.
\newblock ESI Lectures in Mathematics and Physics. European Mathematical
  Society (EMS), Z\"urich, 2010.

\bibitem{Wie84}
{\sc Wiegerinck, J. J. O.~O.}
\newblock Domains with finite-dimensional {B}ergman space.
\newblock {\em Math. Z. 187}, 4 (1984), 559--562.

\end{thebibliography}

\end{document}